\documentclass[11pt]{amsart}

\usepackage{amsthm,amssymb,amsfonts,amsmath,enumerate,graphicx}
\newcommand{\Z}{\mathbb{Z}}
\newcommand{\R}{\mathbb{R}}
\newcommand{\K}{{\mathcal K}}

\renewcommand{\P}{{\mathcal P}}

\newcommand\fl[1]{\left\lfloor {#1} \right\rfloor} 
\newcommand\fr[1]{\left\{ {#1} \right\}} 
\newcommand\saw[1]{\left(\!\left( #1 \right)\!\right)}
\newcommand\ded{{\rm s}}
\newcommand\car{{\rm c}}
\renewcommand\r{{\rm r}}
\newcommand\RC{{\rm R}}
\newcommand\ehr{\operatorname{ehr}}

\def\v{{\boldsymbol v}}

\def\p{\lceil p \rceil}
\def\q{\lceil q \rceil}

\newtheorem{theorem}{Theorem}
\newtheorem{corollary}[theorem]{Corollary}

\newtheorem{lemma}[theorem]{Lemma}

\newtheorem*{definition}{Definition}

\newcommand\comment[1]{}                
\renewcommand\comment[1]{\emph{[#1]}}           
\newcommand{\rem}[1]{} 
\begin{document}
\title{Rademacher--Carlitz Polynomials}

\author{Matthias Beck}
\address{Department of Mathematics\\
         San Francisco State University\\
         San Francisco, CA 94132, USA}
\email{mattbeck@sfsu.edu}

\author{Florian Kohl}
\address{Department of Mathematics  \\
  University of Kentucky \\
  Lexington, KY 40506, USA}
\email{florian.kohl@uky.edu}

\begin{abstract}
We introduce and study the \emph{Rademacher--Carlitz polynomial}
\[
  \RC(u, v, s, t, a, b) := \sum_{ k = \lceil s \rceil }^{ \lceil s \rceil + b - 1 } u^{ \fl{ \frac{ ka + t }{ b } } } v^k 
\]
where $a, b \in \Z_{ >0 }$, $s, t \in \R$, and $u$ and $v$ are variables.
These polynomials generalize and unify various Dedekind-like sums and polynomials; most naturally, one may view $\RC(u, v, s, t, a, b)$ as a polynomial analogue (in the sense of Carlitz) of the \emph{Dedekind--Rademacher sum}
\[
  \r_t(a,b) := \sum_{k=0}^{b-1}\left(\left(\frac{ka+t}{b} \right)\right) \left(\left(\frac{k}{b} \right)\right) ,
\]
which appears in various number-theoretic, combinatorial, geometric, and computational contexts.
Our results come in three flavors: we prove a reciprocity theorem for Rademacher--Carlitz polynomials, we show how they are the only nontrivial ingredients of integer-point transforms
\[
  \sigma(x,y):=\sum_{(j,k) \in \mathcal{P}\cap \Z^2} x^j y^k
\]
of any rational polyhedron $\mathcal{P}$, and we derive a novel reciprocity theorem for Dedekind--Rademacher sums, which follows naturally from our setup.
\end{abstract}



\date{1 October 2013}

\thanks{M.\ Beck's research was partially supported by the NSF (DMS-1162638).}

\maketitle

\renewcommand{\labelitemi}{}

\section{Introduction}\label{introsection}

While studying the transformation properties of
$\eta (z) := e^{ \pi i z / 12 } \prod_{ n \geq 1 } \left( 1 - e^{ 2 \pi i n z } \right) $
under $ \mbox{SL}_{2} ( \Z ) $, Richard Dedekind, in the 1880's \cite{dedekind}, naturally arrived at
what we today call the \emph{Dedekind sum}
\[
  \ded \left( a, b \right) := \sum_{ k=0 }^{ b-1 } \saw{ \frac{ ka }{ b } } \saw{ \frac{ k }{ b } } ,
\]
where $a$ and $b$ are positive integers and
\[
  \saw x :=
  \begin{cases}
    x - \fl x - \frac 1 2 & \text{ if } x \notin \Z , \\
    0 & \text{ if } x \in \Z .
  \end{cases}
\]
The Dedekind sum and its generalizations have since intrigued mathematicians from various areas
such as analytic (see, e.g., \cite{almkvist,bayadsimsek}) and algebraic number theory
(see, e.g., \cite{charollois,ota,solomondedsum}), topology (see, e.g., \cite{hirzebruchzagier,meyersczech}),
algebraic (see, e.g., \cite{brion,gunnelsczech,pommersheim}) and combinatorial geometry (see, e.g., \cite{ccd,mordell}), and algorithmic
complexity (see, e.g.,~\cite{knuth}).

Almost a century after the appearance of Dedekind sums, Leonard Carlitz introduced a
polynomial analogue, the \emph{Dedekind--Carlitz polynomial}
\[
  \car \left( u, v, a, b \right) := \sum_{ k=1 }^{ b-1 } u^{ \fl{ \frac{ ka }{ b } } } v^{ k-1 } .
	\]
Here $u$ and $v$ are indeterminates and $a$ and $b$ are positive integers.
Undoubtedly the most important basic property for any Dedekind-like sum is \emph{reciprocity}. For
the Dedekind--Carlitz polynomials, it says that if $a$ and $b$ are relatively prime then \cite{carlitzpolynomials}
\begin{equation}\label{carlitzreciprocity}
  \left( v-1 \right) \, \car \left( u,v, a,b \right) +  \left( u-1 \right) \, \car \left( v,u,b,a \right) = u^{ a-1 } v^{ b-1 } -1\, .
\end{equation}

Carlitz's reciprocity theorem generalizes that of Dedekind \cite{dedekind}, which states that for
relatively prime positive integers $a$ and $b$,
\begin{equation}\label{eq:dedrec}
  \ded \left( a,b \right) + \ded \left( b,a \right) = - \frac{ 1 }{ 4 }+ \frac{ 1 }{ 12 } \left(
\frac{ a }{ b } + \frac{ 1 }{ ab } + \frac{ b }{ a } \right) .
\end{equation}
Dedekind reciprocity follows from \eqref{carlitzreciprocity} by applying the operators $u \,
\partial{u}$ twice and $v\, \partial v$ once to Carlitz's reciprocity identity.

Dedekind sums have many generalizations. One of the earliest will play a central role in this paper:
for $a, b \in \Z_{ >0 }$, and $t \in \R$, we define the \emph{Dedekind--Rademacher
sum}~\cite{rademacherdedekind}
\begin{equation}\label{eq:raddeddef}
  \r_t(a,b) := \sum_{k=0}^{b-1}\left(\left(\frac{ka+t}{b} \right)\right) \left(\left(\frac{k}{b} \right)\right) .
\end{equation}
Our goal is to introduce and study an analogue of this sum in the world of polynomials: for $a, b \in \Z_{ >0 }$, $s, t \in \R$, and variables $u$ and $v$, we define the
\emph{Rademacher--Carlitz polynomial}
\[
  \RC(u, v, s, t, a, b) := \sum_{ k = \lceil s \rceil }^{ \lceil s \rceil + b - 1 } u^{ \fl{ \frac{
ka + t }{ b } } } v^k . 
\]
Naturally, Dedekind--Carlitz polynomials are special cases of Rademacher--Carlitz polynomials, in the
sense that $v \, \car (u, v, a, b) = \RC(u, v, 0, 0, a, b) - 1$.
It will be handy to abbreviate the linear function $f(x) := \frac{ ax + t }{ b }$ which appears
in the exponent of $u$, and so we will typically use the notation
\[
  \RC(u, v, s, f) := \sum_{ k = \lceil s \rceil }^{ \lceil s \rceil + b - 1 } u^{ \fl{ f(k) } } v^k 
\]
with the understanding that $b$ equals the denominator in the linear function $f$.

Our motivation to study Rademacher--Carlitz polynomials is twofold: first, they seem natural
generalizations of Dedekind--Carlitz polynomials and, as we will see below, they give rise not only to
new reciprocity theorems but also new results on old constructs, such as Dedekind--Rademacher sums. Our
second motivation stems from the fact that Rademacher--Carlitz polynomials appear naturally---as we
will also show below---in the \emph{integer-point transforms}
\[
  \sigma_\P (x,y) := \sum_{ (m,n) \in \P \cap \Z^2 } x^m y^n
\]
of 2-dimensional rational polyhedra $\P$, in particular, 2-dimensional cones/polygons with rational
vertices. In fact, our paper extends some of the methods introduced in \cite{bhm}, which showed that
Dedekind--Carlitz polynomials are natural ingredients for 2-dimensional \emph{lattice} polyhedra, i.e.,
those with integral vertices. Carlitz's reciprocity theorem \eqref{carlitzreciprocity} was a natural by-product of the
geometric approach of \cite{bhm}, and our first result, which mirrors the geometric setup of
\cite{bhm}, is a reciprocity theorem for Rademacher--Carlitz polynomials.

\begin{theorem}\label{reciprocityrcpolyfinal}
Let $f(x) := \frac{ ax + t }{ b }$ be a linear function with relatively prime $a, b \in \Z_{ >0 }$, $t \in \R$, and let $(p,q) \in \R^2$ be a point on the graph of $f$. Then
\[
  v(1-u) \, \RC \left( v,u,p,f \right) + u(1-v) \, \RC \left( u,v,q,f^{-1} \right)
  = u^{ \p } v^{ \q } \left( 1 - u^b v^a \right) - u^c v^d (1-u) (1-v) \, ,
\]
where $(c,d)$ is the unique lattice point on the half-open line segment $\left[(p,q), (p+b, q+a) \right)$; 
if there are no integer points on the graph of $f$ (and so $(c,d)$ does not exist), the last term on the 
right-hand side needs to be omitted.
\end{theorem}

We give a proof in Section \ref{RadCarReci}, where we will also show how \eqref{carlitzreciprocity} follows as a corollary.
One can phrase the conditions in Theorem \ref{reciprocityrcpolyfinal} in purely number-theoretic terms as follows.

\begin{corollary}
Let $a, b \in \Z_{ >0 }$ be relatively prime and $p, q \in \R$. Then
\begin{align*}
  &v(1-u) \, \RC \left( v,u,p,bq-ap,a,b \right) + u(1-v) \, \RC \left( u,v,q,ap-bq,b,a \right) \\
  &\qquad = u^{ \p } v^{ \q } \left( 1 - u^b v^a \right) - u^c v^d (1-u) (1-v) \, ,  
\end{align*}
where $c \in \Z$ is (uniquely) determined by the conditions
\[
  a c \equiv ap-bq \pmod b
  \qquad \text{ and } \qquad
  p \le c < p+b \, ,
\]
and $d := \frac{ ac + bq-ap }{ b }$. If $ap-bq \notin \Z$ then the last term on the right-hand side needs to be omitted.
\end{corollary}

Returning to our second motivation, we remark that the evaluation $\sigma_\P(1,1)$ of an integer-point
transform yields the number of integer lattice points in $\P$. Ehrhart \cite{ehrhartpolynomial}
famously proved in the 1960s that the counting function
\[
  \ehr_\P(t) := \# \left( t \P \cap \Z^d \right) 
\]
is a polynomial in the positive integer variable $t$ when $\P$ is a lattice polytope, and a
quasipolynomial when $\P$ is a rational polytope (see, e.g., \cite{ccd} for more on Ehrhart
quasipolynomials).
It is a natural question how to compute Ehrhart (quasi-)polynomials and integer-point transforms, both
in a computational complexity sense and in terms of ingredients for possible formulas. We will only
briefly touch on the computational aspect, which is governed by Barvinok's theorem \cite{barvinokalgorithm}. 
The ingredients of degree-2 Ehrhart polynomials are easy; they essentially follow from Pick's theorem
\cite{pick} (of which Ehrhart's theorem can be viewed as a far-reaching generalization). The
classification question for degree-2 Ehrhart \emph{quasi}polynomials, i.e., stemming from rational
polygons was answered much more recently \cite{polygons}; here Dedekind--Rademacher sums play a crucial
role as the only nontrivial ingredients.
The analogous classification question for integer-point transforms of lattice polygons was answered in
\cite{bhm}, and Dedekind--Carlitz polynomials played the role here of the nontrivial ingredients. Our
next result provides formulas for the integer-point transforms of \emph{rational} polygons; it can be viewed as a common
generalization (and combination) of the classification results in \cite{bhm} and \cite{polygons}, and
indeed, from this point of view, it should come as no surprise that Rademacher--Carlitz polynomials make an appearance.

\begin{theorem}\label{ipttriangle}
Let $a, b, c, d, e, f, g, h \in \Z_{ > 0 }$, and let $\Delta$ denote the triangle with vertices
$(\frac{e}{f},\frac{g}{h})$, $(\frac{a}{b},\frac{g}{h})$ and $(\frac{e}{f},\frac{c}{d})$. Moreover, we
define $\alpha:=dh(be-af)$, $\beta:=bf(ch-dg)$, and $l(x):=\frac{\beta}{\alpha}x + \frac{c}{d}-
\frac{e\alpha}{f\beta}$.
Then the integer-point transform of $\Delta$ equals
\[
  \sigma_{\Delta}(x,y)
  = \frac{x^{\lceil \frac{a}{b} \rceil}y^{ \lceil \frac{g}{h} \rceil}}{(1-x)(1-y)} 
  + \frac{ \RC \bigl( x, y, \frac g h, l^{ -1 } \bigr) }{(1-x^{-1})(1-x^{\alpha}y^{\beta})}
  + \frac{ \RC \bigl( y, x, \frac e f, l \bigr) }  {(1-y^{-1})(1-x^{-\alpha}y^{-\beta})} \, .
\]
\end{theorem}
We give a proof in Section \ref{IPTratpolygon}.
Theorem \ref{ipttriangle} suffices to provide formulas for the integer-point transform of \emph{any} rational polygon: we can triangulate a
given rational polygon, hence we only have to treat the case of rational triangles and rational line
segments, whose integer-point transforms are relatively straightforward to compute. Using a simple
geometric argument (which we will see in Section \ref{IPTratpolygon}), we can reduce the case of
rational triangles to rational \emph{right} triangles with edges parallel to $x$- and $y$-axis, which are the contents of Theorem~\ref{ipttriangle}.

Our final result is a pleasant by-product of the geometric treatment of Dedekind-like sums; it
turns out that we obtain the following reciprocity theorem for Dedekind--Rademacher sums, which seems
to be new.

\begin{theorem}\label{thm:newrad}
Let $a$ and $b$ be relatively prime positive integers with $a < b$, and let $t \in \R$ with $0 \le t < b$. Then
\begin{align*}
  &\r_{ -t }(a,b) + \r_{ t } (b,a) \, = \\
  &\qquad \frac{ 1 }{ 12 } \left( \frac a b + \frac{ 1 }{ ab } + \frac b a \right) - \frac{1}{4} + \frac{ 1 }{2 a b} \lfloor t \rfloor \left( \lfloor t \rfloor + 1 \right)- \frac 1 2 \left\lfloor \frac t a \right \rfloor
  - \frac \chi 2 \left( \saw{\frac{a^{-1} t}{b}} + \saw{\frac{b^{-1} t}{a}} \right) ,
\end{align*}
where $\chi$ equals $1$ or $0$ depending on whether or not $t \in \Z$,
$a \, a^{ -1 } \equiv 1 \bmod b$, and $b \, b^{ -1 } \equiv 1 \bmod a$.
\end{theorem}

Note that the conditions on $a$, $b$, and $t$ do not constitute a restriction for practical purposes, as
\[
  r_t(a,b) = r_{ t \bmod b } (a \bmod b, b) \, .
\]
At any rate, our proof of Theorem \ref{thm:newrad}, which we give in Section~\ref{DedRadReci}, contains reformulations without the conditions $a<b$ and $0 \le t < b$.

Dedekind's reciprocity theorem \eqref{eq:dedrec} follows naturally from Theorem \ref{thm:newrad} by
setting $t=0$.
However, the more interesting comparison is with Rademacher's reciprocity theorem, which he stated as follows \cite{rademacherdedekind}:
For $ a,b \in \Z $ and $ x,y \in \R $, let
\begin{equation}\label{eq:radorigdef}
  \ded (a,b;x,y) := \sum_{ k=0 }^{ b-1 } \saw{ \frac{ (k+y) a }{ b } + x } \saw{ \frac{ k+y }{ b } } .
\end{equation}
Then, if $a$ and $b$ are relatively prime and $x$ and $y$ are not both integers,
  \[ \ded (a,b;x,y) + \ded (b,a;y,x) = ((x)) ((y)) + \frac{ 1 }{ 2 } \left( \frac{ a }{ b } B_{ 2 } (y) + \frac{ 1 }{ ab } B_{ 2 } (ay+bx) + \frac{ b }{ a } B_{ 2 } (x) \right) , \] 
where $ B_{ 2 } (x) := \fr{x}^{ 2 } - \fr x + \frac{ 1 }{ 6 } $ is the periodized second Bernoulli polynomial. 
A moment's thought reveals that any sum of the form \eqref{eq:raddeddef} can be expressed in the form \eqref{eq:radorigdef} and vice versa.
Indeed, setting $y=0$ and $x = \frac t b$ gives
\[
  \ded \left( a,b; \frac t b , 0 \right) = \sum_{ k=0 }^{ b-1 } \saw{ \frac{ ka + t }{ b } } \saw{ \frac{ k }{ b } }
  \quad \text{ and } \quad
  \ded \left( b,a; 0, \frac t b \right) = \sum_{ k=0 }^{ a-1 } \saw{ \frac{ kb + t }{ a } } \saw{ \frac{ k + \frac t b }{ b } } .
\]
The latter sum equals $\sum_{ k=0 }^{ a-1 } \saw{ \frac{ kb + t }{ a } } \saw{ \frac{ k }{ b } }$ plus some trivial terms.
So Rademacher's reciprocity theorem expressed in terms of $\r_t(a,b)$ says that
\[
  \r_t(a,b) + \r_t(b,a)
\]
equals a simple expression in terms of $a$, $b$, and $t$.
Theorem \ref{thm:newrad}, on the other hand, says that
\[
  \r_t(a,b) + \r_{-t}(b,a)
\]
equals a simple expression, and so it gives a statement complementary to Rademacher reciprocity.
As far as we can tell, the only overlap of the two reciprocity theorems is the case $t=0$, i.e.,
Dedekind's reciprocity theorem.


\section{The reciprocity theorem for Rademacher--Carlitz polynomials} \label{RadCarReci}

\begin{proof}[Proof of Theorem \ref{reciprocityrcpolyfinal}]
As mentioned in the introduction, we follow the ideas of \cite{bhm} which gave a novel geometric proof of \eqref{carlitzreciprocity}.
Let $f(x) := \frac{ ax + t }{ b }$ with $a, b \in \Z_{ >0 } $, where $\gcd(a,b) = 1$, and $t \in \R$, and let $(p,q) \in \R^2$ be a
point on the graph of $f$.
Consider the half-open cones
\begin{align*}
  \K_1 &:= \left\{ (p,q) + \lambda_1 (1,0) + \lambda_2 (b,a) : \, \lambda_1 > 0, \ \lambda_2 \ge 0
\right\} \\
  \K_2 &:= \left\{ (p,q) + \lambda_1 (0,1) + \lambda_2 (b,a) : \, \lambda_1 > 0, \ \lambda_2 \ge 0
\right\}
\end{align*}
and the ray
\[
  \K_3 := \left\{ (p,q) + \lambda (b,a) : \, \lambda \ge 0 \right\} . 
\]
\begin{figure}[h]
\ifx\JPicScale\undefined\def\JPicScale{0.6}\fi
\unitlength \JPicScale mm
\begin{picture}(144.47,88.68)(0,0)
\put(122.24,75.66){\makebox(0,0)[cc]{}}

\linethickness{0.3mm}
\put(57.89,9.47){\line(0,1){79.21}}
\put(57.89,88.68){\vector(0,1){0.12}}
\linethickness{0.3mm}
\put(15.17,41.41){\line(1,0){128.73}}
\put(143.9,41.41){\vector(1,0){0.12}}
\linethickness{0.3mm}
\put(43.09,9.47){\line(0,1){79.21}}
\linethickness{0.3mm}
\put(14.61,30.62){\line(1,0){129.87}}
\put(31.12,26.48){\makebox(0,0)[cc]{$(P,Q)$}}

\put(120.55,62.35){\makebox(0,0)[cc]{}}

\put(133.65,59.45){\makebox(0,0)[cc]{}}

\put(128.91,63.18){\makebox(0,0)[cc]{$\mathcal{K}_1$}}

\linethickness{0.3mm}
\multiput(43.08,30.62)(0.23,0.12){439}{\line(1,0){0.23}}
\linethickness{0.3mm}
\multiput(63.03,30.63)(0.22,0.12){137}{\line(1,0){0.22}}
\linethickness{0.3mm}
\multiput(43.1,44.31)(0.23,0.12){135}{\line(1,0){0.23}}
\linethickness{0.3mm}
\put(43.1,30.19){\line(0,1){13.7}}
\linethickness{0.3mm}
\multiput(74.45,47.04)(0,2.07){7}{\line(0,1){1.03}}
\linethickness{0.3mm}
\put(43.65,30.62){\line(1,0){19.38}}
\linethickness{0.3mm}
\multiput(74.45,47.04)(1.98,0){10}{\line(1,0){0.99}}
\put(84.67,36.02){\makebox(0,0)[cc]{$\Pi_1$}}

\put(50.79,55){\makebox(0,0)[cc]{$\Pi_2$}}

\put(99.19,24.2){\makebox(0,0)[cc]{}}

\put(81.82,71.47){\makebox(0,0)[cc]{$\mathcal{K}_2$}}


\end{picture}
\caption{The shifted first quadrant split into two pointed cones \rem{\bf [$\Pi_1$ is half open the wrong
way]}}\label{fig:conicdecomp}
\end{figure}
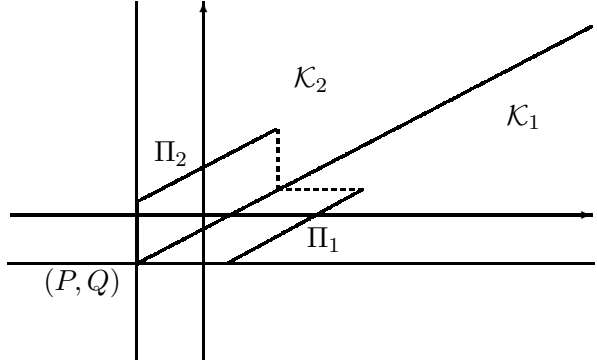
These three objects give a disjoint conic decomposition of the shifted first quadrant, shown in
Figure~\ref{fig:conicdecomp}:
\begin{equation}\label{eq:conicdecomp}
  \left\{ (p,q) + \lambda_1 (1,0) + \lambda_2 (0,1) : \, \lambda_1 , \lambda_2 \ge 0 \right\} = \K_1 \cup \K_2 \cup \K_3 \, ,
\end{equation}
and our goal is to compute the integer-point transforms on both sides. For the shifted
first quadrant, this integer-point transform is
\[
  \frac{ u^{ \p } v^{ \q } }{ (1-u) (1-v) } \, .
\]
By a simple tiling argument (see, for example, \cite[Chapter~3]{ccd}), the integer-point transform $\sigma_{ \K_1 } (u,v)$ of the half-open cone $\K_1$ is
\[
  \sigma_{ \K_1 } (u,v) = \frac{ \sigma_{ \Pi_1 } (u,v) }{ \left( 1 - u \right) \left( 1 - u^b v^a \right) } 
\]
where
\[
  \Pi_1 := \left\{ (p,q) + \lambda_1 (1,0) + \lambda_2 (b,a) : \, 0 < \lambda_1 \le 1, \ 0 \le
\lambda_2 < 1 \right\} ,
\]
the \emph{fundamental parallelogram} of $\K_1$.
Since it has width 1, there is exactly one integer point in $\Pi_1$ for each $y$ running from
$\left\lceil q \right\rceil$ to $\left\lceil q \right\rceil + a - 1$. The $x$-coordinate of this
integer point is $\left\lfloor f^{ -1 } (y) \right\rfloor + 1$. Thus
\[
  \sigma_{\Pi_1}(u,v) 
  = \sum_{k=\q}^{\q+a-1 }{u^{\lfloor f^{-1}(k)\rfloor +1}v^{k}}
  = u \, \RC \left( u,v,q,f^{-1} \right) .
\]
With a similar argument, changing the roles of the axes, we obtain our second integer-point transform:
\[
  \sigma_{ \K_2 } (u,v) = \frac{ \sigma_{ \Pi_2 } (u,v) }{ \left( 1 - v \right) \left( 1 - u^b v^a \right) } 
\]
where
\[
  \sigma_{\Pi_2}(u,v)
  = \sum_{k=\p}^{\p + b - 1} u^k v^{\lfloor f(k) \rfloor + 1}
  = v \, \RC \left( v,u,p,f \right) .
\]
It remains to compute the integer-point transform of the ray $\K_3$. It is clear that any two
lattice points on $\K_3$ differ by a multiple of $(b,a)$ and so
\[
  \sigma_{ \K_3 } (u,v) = \frac{ u^c v^d }{ 1 - u^b u^a }
\]
where $(c,d)$ is the lattice point on $\K_3$ with the smallest coordinates, if there is a lattice
point on $\K_3$ at all---otherwise $\sigma_{ \K_3 } (u,v)$ will simply not appear in our formulas.

Thus \eqref{eq:conicdecomp} translates into the identity of rational generating functions
\[
  \frac{ u^{ \p } v^{ \q } }{ (1-u) (1-v) }
  = \frac{ u \, \RC \left( u,v,q,f^{-1} \right) }{ \left( 1 - u \right) \left( 1 - u^b v^a \right) } 
  + \frac{ v \, \RC \left( v,u,p,f \right) }{ \left( 1 - v \right) \left( 1 - u^b v^a \right) } 
  + \frac{ u^c v^d }{ 1 - u^b v^a } \, ,
\]
where the last term only appears if $\K_3$ contains lattice points.
Clearing denominators gives Theorem~\ref{reciprocityrcpolyfinal}.
\end{proof}

Carlitz's reciprocity theorem \eqref{carlitzreciprocity} follows as an immediate corollary by choosing $t=p=q=0$:
note that then $c=d=0$, and so Theorem~\ref{reciprocityrcpolyfinal} gives in this special case
\[
  v(1-u) \, \RC \left( v,u,0,f \right) + u(1-v) \, \RC \left( u,v,0,f^{-1} \right)
  = 1 - u^b v^a - (1-u) (1-v) \, .
\]
We rewrite the expression on the left to see Dedekind--Carlitz polynomials appear:
\begin{align*}
  &v(1-u) \left( \RC \left( v,u,0,f \right) - 1 \right) + u(1-v) \left( \RC \left( u,v,0,f^{-1} \right) - 1
\right) \\
  &\qquad = 1 - u^b v^a - (1-u) (1-v) - v(1-u) - u(1-v) \\
  &\qquad = - u^b v^a + uv \, . 
\end{align*}
Dividing by $-uv$ gives~\eqref{carlitzreciprocity}.

We finish this section with a remark about computational complexity. In the introduction we hinted at Barvinok's theorem \cite{barvinokalgorithm}, which says that in fixed dimensions,
the integer-point transform $\sigma_{\mathcal{P}}(x_1, \dots, x_d)$ of a rational polyhedron $\mathcal{P}$ can be computed as a sum of short rational functions in $x_1$, $x_2$,
$\dots$, $x_d$ in time polynomial in the input size of $\mathcal{P}$. Thus (say) $\sigma_{\Pi_2}(u,v)$ can be computed efficiently, which means we can compute Rademacher--Carlitz sums
efficiently. (This is a nontrivial statement, since Rademacher--Carlitz sums have exponentially many terms when measured in the input size of its parameters.) 


\section{Integer-point transforms of rational polygons}\label{IPTratpolygon}

In this section, we give the details behind our claim that Theorem \ref{ipttriangle} suffices to
characterize the integer-point transform of any rational polygon, and we will prove Theorem~\ref{ipttriangle}.

As mentioned in the introduction, any rational polygon can be triangulated, and so we can compute its
integer-point transform in an inclusion-exclusion fashion from integer-point transforms of rational
line segments and rational triangles. Furthermore, we can embed an arbitrary triangle in a rectangle in
such a way that we can express the triangle as a set union/subtraction of rectangles and right
triangles with edges parallel to $x$- and $y$-axis, as suggested by Figure \ref{fig:triangles2}; if the triangle was rational to begin with, so will be the rectangles and right triangles.
\begin{figure}[h]
	\ifx\JPicScale\undefined\def\JPicScale{0.50}\fi
\unitlength \JPicScale mm
\begin{picture}(147.67,77.11)(0,0)
\linethickness{0.3mm}
\put(13.16,77.11){\line(1,0){51.43}}
\put(13.16,12.5){\line(0,1){64.6}}
\put(64.58,12.5){\line(0,1){64.6}}
\put(13.16,12.5){\line(1,0){51.43}}
\linethickness{0.3mm}
\put(13.16,77.1){\line(1,0){51.43}}
\put(13.16,12.49){\line(0,1){64.6}}
\put(64.59,12.49){\line(0,1){64.6}}
\put(13.16,12.49){\line(1,0){51.43}}
\linethickness{0.3mm}
\put(96.22,76.72){\line(1,0){51.43}}
\put(96.22,12.11){\line(0,1){64.6}}
\put(147.64,12.11){\line(0,1){64.6}}
\put(96.22,12.11){\line(1,0){51.43}}
\linethickness{0.3mm}
\put(96.22,76.71){\line(1,0){51.43}}
\put(96.22,12.11){\line(0,1){64.6}}
\put(147.65,12.1){\line(0,1){64.6}}
\put(96.22,12.11){\line(1,0){51.43}}
\linethickness{0.3mm}
\multiput(13.16,12.49)(0.12,0.15){429}{\line(0,1){0.15}}
\linethickness{0.3mm}
\multiput(56.36,37.35)(0.12,0.58){69}{\line(0,1){0.58}}
\linethickness{0.3mm}
\multiput(13.16,12.49)(0.21,0.12){207}{\line(1,0){0.21}}
\linethickness{0.3mm}
\put(56.36,12.49){\line(0,1){24.85}}
\linethickness{0.3mm}
\put(56.36,37.34){\line(1,0){8.23}}
\linethickness{0.3mm}
\multiput(96.22,12.11)(0.12,0.22){293}{\line(0,1){0.22}}
\linethickness{0.3mm}
\multiput(131.44,76.34)(0.12,-0.23){135}{\line(0,-1){0.23}}
\linethickness{0.3mm}
\multiput(96.22,12.11)(0.19,0.12){274}{\line(1,0){0.19}}
\end{picture}
	\label{fig:triangles2}
	\caption{Triangles embedded in a rectangle and right triangles.}
\end{figure}
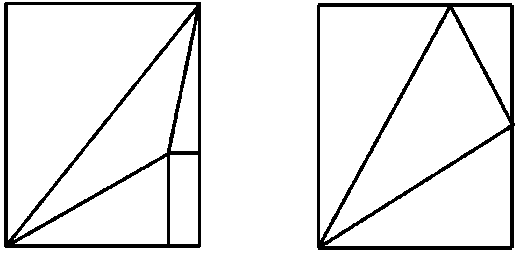
The integer-point transforms of rectangles are easy, and thus it remains to compute integer-point
transforms of right triangles with edges parallel to $x$- and $y$-axis, which (by a harmless lattice
transformation) we may assume to be in the first quadrant with its right angle in the southwestern
vertex. That is, it remains to prove Theorem~\ref{ipttriangle}. 
\begin{proof}[Proof of Theorem~\ref{ipttriangle}]
As stated in the conditions, we assume that $\Delta$ looks like in Figure \ref{fig:triangles2}.
\begin{figure}[h]
\ifx\JPicScale\undefined\def\JPicScale{0.5}\fi
\unitlength \JPicScale mm
\begin{picture}(130.53,83.95)(0,0)
\put(70,81.84){\makebox(0,0)[cc]{$\left(\frac{e}{f},\frac{c}{d}\right)$}}

\linethickness{0.3mm}
\put(32.75,3.95){\line(0,1){80}}
\put(32.75,83.95){\vector(0,1){0.12}}
\linethickness{0.3mm}
\put(20.53,18.98){\line(1,0){110}}
\put(130.53,18.98){\vector(1,0){0.12}}
\linethickness{0.3mm}
\put(56.55,34.07){\line(0,1){44.08}}
\linethickness{0.3mm}
\put(56.55,34.07){\line(1,0){48.89}}
\linethickness{0.3mm}
\multiput(56.55,78.15)(0.13,-0.12){367}{\line(1,0){0.13}}
\put(114.5,41.71){\makebox(0,0)[cc]{$\left(\frac{a}{b},\frac{g}{h}\right)$}}

\put(53.01,27.88){\makebox(0,0)[cc]{$\left(\frac{e}{f},\frac{g}{h}\right)$}}

\end{picture}

\caption{The rational right triangle from Theorem \ref{ipttriangle}. \rem{\bf [Again, a bit much
white space... plus the labels have too small font]}}
\end{figure}
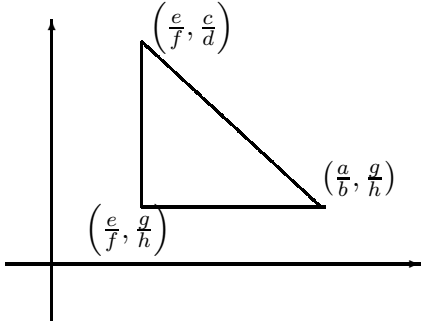
To compute the integer-point transform of $\Delta$, we use Brion's theorem \cite{brion}, which says that
$\sigma_\Delta(x,y)$ equals the sum of the integer-point transforms of the three vertex cones of
$\Delta$.
(The \emph{vertex cone} of a polytope $\P$ at a vertex $\v$ is the smallest cone with apex $\v$ that contains $\P$.)
Thus we need to compute the integer-point transforms of the vertex cones
\begin{align*}
  V_{1}& := \left\{ (\tfrac{e}{f},\tfrac{g}{h}) + \lambda_{1} (1,0)+ \lambda_{2}(0,1) : \, \lambda_{1},  \lambda_{2} \ge 0 \right\} \\
  V_{2} &:= \left\{ (\tfrac{a}{b},\tfrac{g}{h}) + \lambda_{1} (-1,0)+ \lambda_{2} \bigl( dh(be-af),
bf(ch-dg) \bigr) : \, \lambda_{1}, \lambda_{2} \ge 0 \right\} \\
  V_{3} &:= \left\{ (\tfrac{e}{f},\tfrac{c}{d}) + \lambda_{1} (0,-1) + \lambda_{2} \bigl( -dh(be-af),
-bf(ch-dg) \bigr) : \, \lambda_{1}, \lambda_{2} \ge 0 \right\} .
\end{align*}
To shorten notation, we define, as in the statement of Theorem~\ref{ipttriangle}, $\alpha:=dh(be-af)$
and $\beta:=bf(ch-dg)$. The integer-point transform of $V_1$ is straightforward:
\begin{equation}\label{eq:sigv1}
\sigma_{V_{1}}(x,y)=\sum_{k\geq \lceil \frac{e}{f} \rceil\mbox{, } j\geq \lceil \frac{g}{h} \rceil}
{x^k y^j}=\frac{x^{\lceil \frac{e}{f} \rceil}y^{ \lceil \frac{g}{h} \rceil}}{(1-x)(1-y)} \, .
\end{equation}
For the other two vertex cones, we use a tiling argument similar to the one in the proof of Theorem \ref{reciprocityrcpolyfinal}.
This gives,
\begin{align}
  \sigma_{V_{2}}(x,y) &= \frac{ \sigma_{\Pi_{2}} (x,y) }{(1-x^{-1})(1-x^{\alpha}y^{\beta})} \label{eq:sigmav2} \\
  \sigma_{V_{3}}(x,y) &= \frac{ \sigma_{\Pi_{3}} (x,y) }{(1-y^{-1})(1-x^{-\alpha}y^{-\beta})} \label{eq:sigmav3}
\end{align}
where
\begin{align*}
  \Pi_{2} &:= \left\{ (\tfrac{a}{b},\tfrac{g}{h}) + \lambda_{1} (-1,0)+ \lambda_{2}(\alpha,\beta) : \, 0\leq \lambda_{1}, \lambda_{2} <1 \right\} \\
  \Pi_{3} &:= \left\{ (\tfrac{e}{f},\tfrac{c}{d}) + \lambda_{1} (0,-1) + \lambda_{2}(-\alpha,-\beta) : \, 0\leq \lambda_{1}, \lambda_{2}<1\right\}
\end{align*}
are the fundamental parallelograms of $V_2$ and $V_3$, respectively.
To compute the integer-point transform of $\Pi_2$, we note that the linear function $l(x):=\frac{\beta}{\alpha}x + \frac{c}{d} - \frac{e\alpha}{f\beta}$ given in the statement of
Theorem \ref{reciprocityrcpolyfinal} describes the line that contains the hypotenuse of $\Delta$.
Since $\Pi_2$ has height $1$ and is half open, for every integral $y$-coordinate between $\lceil \frac{g}{h} \rceil$ and $\lceil\frac{g}{h}\rceil+\beta-1$ there is exactly one
point with integral $x$-coordinate, namely $\lfloor l^{-1}(y)\rfloor$, and so
\begin{equation}\label{eq:sigmap2}
  \sigma_{\Pi_{2}}(x,y)
  = \sum_{k=\left\lceil \frac{g}{h} \right\rceil}^{\left\lceil\frac{g}{h}\right\rceil+\beta-1} x^{\left\lfloor l^{-1}(k)\right\rfloor}y^k
  = \RC \left( x, y, \frac g h, l^{ -1 } \right) .
\end{equation}
A parallel argumentation yields
\begin{equation}\label{eq:sigmap3}
  \sigma_{ \Pi_3 } (x,y) 
  = \sum_{k=\left\lceil \frac{e}{f}\right\rceil}^{\left\lceil \frac{e}{f}\right\rceil + \alpha - 1}{x^k y^{\lfloor l(k)\rfloor}}
  = \RC \left( y, x, \frac e f, l \right) .
\end{equation}
Brion's theorem says
\[
  \sigma_\Delta (x,y) = \sigma_{ V_1 } (x,y) + \sigma_{ V_2 } (x,y) + \sigma_{ V_3 } (x,y) \, , 
\]
which, using \eqref{eq:sigv1}--\eqref{eq:sigmap3}, yields Theorem~\ref{reciprocityrcpolyfinal}.
\end{proof}


\section{A novel reciprocity theorem for Dedekind--Rademacher sums } \label{DedRadReci}

Our goal in this section is to prove Theorem \ref{thm:newrad}. 
We will need a few identities that are slightly technical but straightforward.
For $x \in \R$ and $m \in \Z_{ >0 }$, we denote by $[x]_m$ the smallest nonnegative real number congruent to $x$ mod~$m$.

\begin{lemma}\label{lem:tech}
Let $a$ and $b$ be positive relatively prime integers, and let $t \in \R$.
\begin{enumerate}[{\rm (a)}]
\item\label{lem:simplesum} 
$ \displaystyle
  \sum_{ k=0 }^{ b-1 } \fr{ \frac{ ak+t }{ b } } = \frac{ b-1 }{ 2 } + \fr t .
$
\item\label{lem:fractorad}
$ \displaystyle
  \sum_{ k=0 }^{ b-1 } k \fr{ \frac{ ak+t }{ b } }
  = b \, \r_t(a,b) + \tfrac{ 1 }{ 4 } b(b-1) + \tfrac 1 2 b \fr t - \tfrac 1 2 [t]_b + \tfrac 1 2
\chi \, b \saw{ \frac{ t a^{ -1 } }{ b } }
$ \\
where $\chi$ equals $1$ or $0$ depending on whether or not $t$ is an integer.
\end{enumerate}
\end{lemma}

\begin{proof}
(\ref{lem:simplesum}) is essentially \emph{Raabe's formula} (see, e.g., \cite[Lemma 1]{grosswald}).

\vspace{6pt}
\noindent
(\ref{lem:fractorad}) We compute
\begin{align*}
  &\frac 1 b \sum_{ k=0 }^{ b-1 } k \fr{ \frac{ ak+t }{ b } }
  = \sum_{ k=1 }^{ b-1 } \fr{ \frac k b } \fr{ \frac{ ak+t }{ b } } \\
  &\qquad = \sum_{ k=1 }^{ b-1 } \saw{ \frac k b } \saw{ \frac{ ak+t }{ b } }
    + \frac 1 2 \sum_{ k=1 }^{ b-1 } \fr{ \frac{ ak+t }{ b } }
    + \frac 1 2 \sum_{ k=1 }^{ b-1 } \fr{ \frac k b }
    - \frac {b-1} 4
    + \frac \chi 2 \saw{ \frac{ t a^{ -1 } }{ b } } \, .
\end{align*}
The last correction term comes from a case-by-case analysis of $\saw{ \frac{ ak+t }{ b } }$: the
argument is an integer if and only if $t$ is an integer congruent to $-ak$ for some integer $k$
between 1 and $b-1$. 
With part (a) and the definition of the Dedekind--Rademacher sum, this becomes
\[
  \frac 1 b \sum_{ k=0 }^{ b-1 } k \fr{ \frac{ ak+t }{ b } }
  = \r_t(a,b)
    + \frac{ b-1 }{ 4 } + \frac 1 2 \fr t - \frac 1 2 \fr{ \frac t b } 
    + \frac \chi 2 \saw{ \frac{ t a^{ -1 } }{ b } } \, .
\]
With $b \fr{ \frac x b } = [x]_b$, this gives~(b).
\end{proof}

\begin{proof}[Proof of Theorem \ref{thm:newrad}]
We start by applying the operators $u\, \partial u$ twice and $v\, \partial v$ once to the identity in Theorem \ref{reciprocityrcpolyfinal}, which yields
\begin{align}
  &2 \sum_{ k = \p }^{ \p + b - 1 } k \fl{ \frac{ ak+t }{ b } }
  +2 \sum_{ k = \p }^{ \p + b - 1 } \!\!k
  +  \sum_{ k = \p }^{ \p + b - 1 } \fl{ \frac{ ak+t }{ b } }
  +  b \nonumber \\
  &\qquad + \sum_{ k = \q }^{ \q + a - 1 } \fl{ \frac{ bk-t }{ a } }^2 
  +2 \sum_{ k = \q }^{ \q + a - 1 } \fl{ \frac{ bk-t }{ a } }
  +  a \label{eq:foursums} \\
 &= ( \p + 2b ) a \p + (2 \p + b) b \q + ab^2 + \chi (2c + 1) \, . \nonumber
\end{align}
Here $\chi$ equals 1 or 0 depending on whether or not there are integer points on the graph of $f(x)
= \frac{ ax+t }{ b }$; since $a$ and $b$ are relatively prime, there will be integer points if and
only if $t \in \Z$, and thus $\chi$ has the same meaning as in Lemma~\ref{lem:tech}.
Recall also from the statement of Theorem \ref{reciprocityrcpolyfinal} that $c$ is the $x$-coordinate of the unique lattice point on the half-open line segment
$\left[(p,q), (p+b, q+a) \right)$. Thus $c \in \Z$ is uniquely determined by the conditions
\[
  c \equiv a^{ -1 } (ap-bq) \pmod b
  \qquad \text{ and } \qquad
  p \le c < p+b \, ,
\]
where $a \, a^{ -1 } \equiv 1 \bmod b$.

There are four nontrivial sums in \eqref{eq:foursums}, which we will uncover now one by one, with the
help of Lemma~\ref{lem:tech}.

\begin{align*}
  &\sum_{ k = \p }^{ \p + b - 1 } k \fl{ \frac{ ak+t }{ b } }
  = \sum_{ k = \p }^{ \p + b - 1 } k \, \frac{ ak+t }{ b }
     - \sum_{ k = \p }^{ \p + b - 1 } k \fr{ \frac{ ak+t }{ b } } \\
  &= \frac 1 3 a b^2 + a b \p + a \p^2 - \frac 1 2 a b - a \p + \frac 1 2 b t + \p t + \frac 1 6 a - \frac 1 2 t 
     - \sum_{ k = 0 }^{ b - 1 } (k + \p) \fr{ \frac{ a (k + \p) + t }{ b } } \\
  &= \frac 1 3 a b^2 + a b \p + a \p^2 - \frac 1 2 a b - a \p + \frac 1 2 b \fl t + \p \fl t + \frac 1 6 a - \frac 1 2 t +\frac 1 2 [a \p+t]_b \\
  &\qquad {} -\frac{1}{4a}b^2 - \frac 1 2 b \p + \frac 1 4 b + \frac 1 2 \p - b \, \r_{ a\p+t } (a,b)
- \frac 1 2 \chi \, b \saw{ \frac{ \p+ta^{ -1 } }{ b } } ,
\end{align*}
where again $a \, a^{ -1 } \equiv 1 \bmod b$.
(Note that $a\p+t \in \Z$ if and only if $t \in \Z$.)

\begin{align*}
  \sum_{ k = \p }^{ \p + b - 1 } \fl{ \frac{ ak+t }{ b } }
  &= \sum_{ k = \p }^{ \p + b - 1 } \frac{ ak+t }{ b }
   - \sum_{ k = \p }^{ \p + b - 1 } \fr{ \frac{ ak+t }{ b } } \\
  &= \frac 1 2 a (b-1) + a \p + t
   - \sum_{ k = 0 }^{ b - 1 } \fr{ \frac{ k+t }{ b } } \\
  &= \frac 1 2 (a-1) (b-1) + a \p + \fl t .
\end{align*}
Analogously,
\[
  \sum_{ k = \q }^{ \q + a - 1 } \fl{ \frac{ bk-t }{ a } }
  = \frac 1 2 (a-1) (b-1) + b \q + \fl{ -t } .
\]
Finally,
\begin{align*}
  &\sum_{ k = \q }^{ \q + a - 1 } \fl{ \frac{ bk-t }{ a } }^2
  = \sum_{ k = \q }^{ \q + a - 1 } \left( \frac{ bk-t }{ a } \right)^2 
    - 2 \sum_{ k = \q }^{ \q + a - 1 } \frac{ bk-t }{ a } \fr{ \frac{ bk-t }{ a } }
    + \sum_{ k = \q }^{ \q + a - 1 } \fr{ \frac{ bk-t }{ a } }^2 \\
  &\qquad = \frac 1 3 a b^2 + b^2 \q - \frac 1 2 b^2 - b t + \frac 1 a \left( b^2 \q^2 - b^2  \q  - 2 b  \q  t +
\frac 1 6 b^2 + b t + t^2 \right) \\
  &\qquad \qquad - \frac{ 2b }{ a } \sum_{ k = \q }^{ \q + a - 1 } k \fr{ \frac{ bk-t }{ a } }
    + \frac{ 2t }{ a } \sum_{ k = \q }^{ \q + a - 1 } \fr{ \frac{ bk-t }{ a } }
    + \sum_{ k = 0 }^{ a - 1 } \fr{ \frac{ k + \fr{ -t } }{ a } }^2 \\
  &\qquad = \frac 1 3 a b^2 + b^2  \q  - \frac 1 2 b^2 - b t + \frac 1 a \left( b^2  \q^2 - b^2  \q  - 2 b  \q  t +
\frac 1 6 b^2 + b t + t^2 \right) \\
  &\qquad \qquad - \frac{ 2b }{ a } \sum_{ k = 0 }^{ a - 1 } (k + \q) \fr{ \frac{ b(k + \q)-t }{ a } }
    + \frac{ 2t }{ a } \sum_{ k = 0 }^{ a - 1 } \fr{ \frac{ k-t }{ a } }
    + \sum_{ k = 0 }^{ a - 1 } \left( \frac{ k + \fr{ -t } }{ a } \right)^2 \\
  &\qquad = \frac 1 3 a b^2 - \frac 1 2 a b + \frac 1 3 a + b^2 \q - b t - \frac 1 2 b^2 + \frac 1 2 b
- \frac 1 2 - \fl{ -t } -b \q\\
   &\qquad \qquad + \frac 1 a \left( b^2 \q^2 - b^2 \q + 2 b \q \fl{-t} + \frac 1 6 b^2 + b t + \frac
1 6 + \fl{ -t }^2 + \fl{ -t } + b\q \right) \\
   &\qquad \qquad - 2b \, \r_{ b \q - t } (b,a) - \frac b a [-a t]_a + \frac b a [b\q - t]_a - \chi
\, b \saw{ \frac{ q - t b^{ -1 } }{ a } } ,
\end{align*}
where $b \, b^{ -1 } \equiv 1 \bmod a$.
(Note that $b\q-t \in \Z$ if and only if $t \in \Z$.)

We are all set to substitute the expressions we found back into \eqref{eq:foursums}.
Simplifying terms such as $\fr t + \fr{-t}$ (which equals 1 if $t \notin \Z$ and 0 if $t \in \Z$) and
$\frac{ [x]_a }{ a } = \fr{ \frac x a }$ gives
\begin{align*}	
  &\r_{ a \p + t }(a,b) + \r_{ b \q - t } (b,a) \, = \ \\
&\quad \frac{a \p^2}{2 b}-\frac{a \p}{2 b}+\frac{b \q^2}{2 a}-\frac{b \q}{2 a}+\frac{b}{12 a}+\frac{a}{12 b}+\frac{1}{12 a b}+\frac{\q \lfloor
   -t\rfloor }{a}+\frac{\q}{2 a}+\frac{\p \lfloor t\rfloor }{b}+\frac{\p}{2 b}+\frac{t}{2 a}-\frac{t}{2 b}\\
&\quad-\p \q+\frac{\p}{2}+\frac{\q}{2}-\frac{3}{4}+\frac{\lfloor -t\rfloor ^2}{2 a b}+\frac{\lfloor
   -t\rfloor }{2 a b}+\frac{1}{2}\fr{\frac{a \p+t}{ b}}+\frac 1 2 \fr{\frac{b \q-t }{ a}}\\
&\quad +\chi\left( -\frac{1}{2}   \saw{\frac{a^{-1} (a \p+t)}{b}}-\frac{1}{2}   \saw{\frac{b^{-1}(b \q-t)}{a}}+\frac{1 }{2}-\frac{c  }{b}\right)
\end{align*}
Now we use the relation $bq = ap + t$, which simplifies the left-hand side to
\[
  \r_{ a \fr{-p} - b\fr{-q} }(a,b) + \r_{ b \fr{-q}- a\fr{-p} } (b,a) \, .
\]
But this means we might as well choose $p$ and $q$ in some interval of length 1; it is easiest to assume $-1 < p, q \le 0$, since this will simplify the right-hand
side most easily:
\begin{align*}
&\r_{ bq-ap }(a,b) + \r_{ ap-bq } (b,a) \, = \ \\
&\qquad \frac{a}{12 b}+\frac{b}{12 a}+\frac{1}{12 a b}-\frac{3}{4}+\frac{\lfloor ap-bq\rfloor ^2}{2 a b}+\frac{\lfloor ap-bq\rfloor }{2 a b}-\frac 1 2 \fl{\frac{ap-bq} {a}}-\frac 1 2 \fl{\frac {bq-ap} b}\\
&\qquad +\chi\left(\frac{1}{2}-\frac{c }{b}-\frac{1}{2}   \saw{\frac{a^{-1} (bq-ap)}{b}}-\frac{1}{2} \saw{\frac{b^{-1}(ap - bq)}{a}} \right) .
\end{align*}
Recall that $c$ is the unique integer satisfying 
\[
  c \equiv a^{ -1 } (ap-bq) \pmod b
  \qquad \text{ and } \qquad
  p \le c < p+b \, ,
\]
Since $-1 < p \le 0$, this condition simply says that $c$ is the smallest nonnegative integer congruent to $a^{ -1 } (ap-bq)$ modulo $b$, that is,
\[
  c = b \fr{ \frac{ a^{ -1 } (ap-bq) }{ b }  } = - b \saw{ \frac{ a^{ -1 } (bq-ap) }{ b }  } + (1-\mu) \, \frac b 2 \, ,
\]
where $\mu = 1$ if $b|bq-ap$ and $\mu = 0$ otherwise. This yields
\begin{align*}
&\r_{ bq-ap }(a,b) + \r_{ ap-bq } (b,a) \, = \ \\
&\qquad \frac{a}{12 b}+\frac{b}{12 a}+\frac{1}{12 a b}-\frac{3}{4}+\frac{\lfloor ap-bq\rfloor ^2}{2 a b}+\frac{\lfloor ap-bq\rfloor }{2 a b}-\frac 1 2 \fl{\frac{ap-bq} {a}}-\frac 1 2 \fl{\frac {bq-ap} b}\\
&\qquad +\chi\left(\frac{ \mu }{2}+\frac{1}{2} \saw{\frac{a^{-1} (bq-ap)}{b}}-\frac{1}{2} \saw{\frac{b^{-1}(ap - bq)}{a}} \right) .
\end{align*}
Now we set $q=0$ and assume that $a < b$, for which the above identity simplifies to
\begin{align*}
\r_{ bq }(a,b) + \r_{ -bq } (b,a) \, = \ 
& \frac{a}{12 b}+\frac{b}{12 a}+\frac{1}{12 a b}-\frac{3}{4}+\frac{\lfloor -bq\rfloor ^2}{2 a b}+\frac{\lfloor -bq\rfloor }{2 a b}-\frac 1 2 \fl{\frac{-bq} {a}}-\frac
1 2 \fl{q}\\
&\qquad +\chi\left(\frac{ \mu }{2}-\frac{1}{2} \saw{\frac{a^{-1} (-bq)}{b}}-\frac{1}{2} \saw{\frac{b^{-1}(-bq)}{a}} \right) ,
\end{align*}
where $\chi$ equals $1$ or $0$ depending on whether or not $bq$ is an integer,
$\mu$ equals $1$ or $0$ depending whether or not $q=0$, 
$a \, a^{ -1 } \equiv 1 \bmod b$, and $b \, b^{ -1 } \equiv 1 \bmod a$.
Noticing that $\fl q = -1$ unless $q=0$, and setting $t = -bq$ (which is a real number in the interval $[0,b)$) yields Theorem~\ref{thm:newrad}.
\end{proof}

We strongly suspect that there exists a more direct proof of Theorem \ref{thm:newrad}. We leave it as a challenge to the reader to find one.

\rem{  &\frac{1}{ab}\left( \frac{ 1 }{ 12 } + \frac 1 2 \fl{ ap-bq } + \frac 1 2 \fl{ ap-bq }^2 \right) -
\p \q + \frac 1 2 \fr{-p} + \frac 1 2 \fr{-q} - \frac 3 4 \\
  &+ \frac{1}{a} \left( \frac{1}{12} b + \frac 1 2 b\q^2 + \frac 1 2 \q + \q \fl{ ap-bq } \right) -
\frac 1 2 \fl{ \frac{ b \fr{-q} }{ a } } \\ 
&+ \frac{1}{b} \left( \frac{1}{12} a + \frac 1 2 a\p^2 + \frac 1 2 \p + \p \fl{ bq-ap } \right) -
\frac 1 2 \fl{ \frac{ a \fr{-p} }{ b } } \\
  &+ \chi \left( \frac 1 2 - \frac c b - \frac 1 2 \saw{\frac{ \fr{-p} }{b} + q a^{ -1 } } - \frac 1 2
\saw{\frac{ \fr{-q} }{a} + p b^{ -1 } } \right)
																					. \qedhere}

\bibliographystyle{amsplain}
\bibliography{bib}

\rem{A look on the left-hand side of the equation reveals that only the fractional parts of $p$ and $q$ contribute to the Dedekind--Rademacher sum. In other words, we can without loss of generality assume that $p$, $q \in [0,1)$. As a consequence, the right-hand side simplifies further, which leads to the following representation.
\begin{theorem}\label{thm:newrad2}
Let $a$ and $b$ be relatively prime positive integers, and let $p, q \in \R$. Then
\begin{align*}
  &\r_{ a \fr{p} }(a,b) + \r_{ b \fr{q} } (b,a) = \\
  &\frac{1}{ab}\left( \frac{ 1 }{ 12 } + \frac 1 2 \fl{ b \fr{q}-a\fr{p }} + \frac 1 2 \fl{ b\fr{q}-a\fr{p} }^2 \right)  - \frac 3 4 + \frac{1}{a} \left( \frac{1}{12} b +\frac 1 2 (a \fr{p} - b \fr{q})\right) \\
  &+ \frac{1}{b} \left( \frac 1 2 + \frac{1}{12} a + \frac 1 2 \fl {b \fr{q} - a \fr{p}}- \frac 1 2 \fr { a \fr{p} - b \fr{q}}  \right)+ \frac 1 2 \fr{ \frac{ a \fr{p} }{ b } } + \frac 1 2 \fr{ \frac{ b \fr{q} }{ a } }\\
  &+ \chi \left( \frac 1 2 - \frac 1 2b - \frac c b - \frac 1 2 \saw{\frac{ \fr{p} }{b} + \fr{q} a^{ -1 } } - \frac 1 2
\saw{\frac{ \fr{q} }{a} + \fr{p} b^{ -1 } } \right) . \qedhere
\end{align*}
where $\chi$ equals $1$ or $0$ depending on whether or not $bq-ap$ is an integer, $a \, a^{ -1 }
\equiv 1 \bmod b$, $b \, b^{ -1 } \equiv 1 \bmod a$, and $c$ is the first coordinate of the smallest
integer solution $(c,d)$ to the equation $by-ax=b\fr q-a \fr p$.
\end{theorem}}
\end{document}